\documentclass[bibalpha]{amsart}
\usepackage{amssymb}
\usepackage{pb-diagram}
\usepackage{comment}
\usepackage{tikz}
\usetikzlibrary{fit}
\usetikzlibrary{patterns}
\usetikzlibrary{positioning}
\usepackage[mathscr]{euscript}

\newcommand{\Cal}[1]{\mathcal #1}

\def \cl {\operatorname{Cl}}

\def \N{\mathbb{N}}

\def \R{\mathbb{R}}

\def \bN{\mathbb{N}}
\def \Z{\mathbb{Z}}

\newcommand{\DSig}{\operatorname{D}_{\Sigma}}
\newcommand{\Int}{\operatorname{Int}}

\newcommand{\Dsig}{\mathrm{D}_{\Sigma}}

\def \cl {\mathrm{Cl}}
\usepackage{enumerate}
\newtheorem{Th}{Theorem}[section]
\newtheorem{Thm}[Th]{Theorem}

\newtheorem{Def}[Th]{Definition}

\newtheorem{Cor}[Th]{Corollary}
\newtheorem{Fact}[Th]{Fact}
\newtheorem{fact}[Th]{Fact}
\newtheorem{Prop}[Th]{Proposition}

\newtheorem*{claim}{Claim}

\newtheorem{Lem}[Th]{Lemma}
\newtheorem{Q}[Th]{Question}
\newtheorem{Qs}[Th]{Questions}
\newtheorem*{Lem*}{Lemma}

\newtheorem*{thmA}{Theorem A}
\newtheorem*{thmB}{Theorem B}

\begin{document}
\title{Fractals and the monadic second order theory of one successor}

\author{Philipp Hieronymi}
\address{Mathematical Institute\\ University of Bonn\\
Endenicher Allee 60\\ 53115 Bonn\\ Germany}
\email{hieronymi@math.uni-bonn.de}

\author{Erik Walsberg}
\address{Department of Mathematics\\ University of California, Irvine}
\email{ewalsber@uci.edu}
\urladdr{https://www.math.uci.edu/\textasciitilde ewalsber}


\maketitle

\begin{abstract}
We show that if $X$ is virtually any classical fractal subset of $\R^n$, then $(\R,<,+,X)$ interprets the monadic second-order theory of $(\N,+1)$.
This result is sharp in the sense that the standard model of the monadic second-order theory of $(\N,+1)$ is known to interpret $(\R,<,+,X)$ for various classical fractals $X$, including the middle-thirds Cantor set and the Sierpinski carpet.
Let $X \subseteq \R^n$ be closed and nonempty.
We show that if the $C^k$-smooth points of $X$ are not dense in $X$ for some $k \ge 1$, then $(\R,<,+,X)$ interprets the monadic second-order theory of $(\N,+1)$.
The same conclusion holds if the packing dimension of $X$ is strictly greater than the topological dimension of $X$ and $X$ has no affine points.
\end{abstract}

\section{Introduction}
\noindent
This paper is a contribution to a larger research enterprise (see \cite{FKMS,FHM,HM,H-TameCantor,HW-Monadic,FHW-Compact,BH-Cantor}) motivated by the following fundamental question:
\begin{center}
\emph{What is the logical/model-theoretic complexity generated by fractal\footnote{We do not work with a precise definition of what a fractal is. Rather we prove our results for sets that have properties characteristic for fractals. For an elegant exposition of this view of fractals, see the introduction in Falconer \cite{Falconer-book}.} objects?}
\end{center}
Here we will focus on fractal objects defined in first-order expansions of the ordered real additive group $(\R,<,+)$. Throughout this paper $\mathcal{R}$ is a first-order expansion of $(\R,<,+)$, and ``definable'' without modification means ``$\mathcal{R}$-definable, possibly with parameters from $\mathbb{R}$''. The main problem we want to address here is:
\begin{center}
\emph{If $\mathcal{R}$ defines a fractal object, what can be said about the logical complexity of $\mathcal{R}$?}
\end{center}
The first result in this direction is \cite[Theorem B] {HW-Monadic}, stating that whenever $\mathcal{R}$ defines a \textbf{Cantor set} (that is, a nonempty compact subset of $\R$ without interior or isolated points), then $\mathcal{R}$
defines an isomorphic copy of the two-sorted first-order structure $(\mathcal{P}(\bN),\bN,\in,+1)$. The latter structure is the standard model of the monadic second-order theory of $(\N,+1)$. Let $\mathcal{B}$ denote this structure. As pointed out in \cite{HW-Monadic}, while the theory of $\mathcal{B}$ is decidable by B\"uchi \cite{Buchi}, the structure does not enjoy any Shelah-style combinatorial tameness properties, such
as NIP or NTP2 (see e.g. Simon \cite{Simon-Book} for definitions). Thus every structure that defines an isomorphic copy of $\mathcal{B}$, can not satisfy these properties either, and for that reason has to be regarded as complicated or
wild in the sense of these combinatorial/model-theoretic tameness notions. In this paper, we extend such results to fractal subsets of $\R^n$. \newline

\noindent Let $X \subseteq \R^n$ be nonempty.
Given $k \geq 0$, a point $p$ on $X$ is \textbf{$\mathbf{C^k}$-smooth} if $U \cap X$ is a $C^k$-submanifold of $\mathbb{R}^n$ for some nonempty open neighbourhood $U$ of $p$.
A point $p$ on $X$ is \textbf{affine} if there is an open neighbourhood $U$ of $p$ such that $U \cap X = U \cap H$ for some affine subspace $H$.
We say that $\mathcal{R}$ is \textbf{field-type} if there is an open interval $I$, definable functions $\oplus, \otimes : I^2 \to I$ such that $(I,<,\oplus,\otimes)$ is isomorphic to $(\mathbb{R},<,+,\cdot)$.

\begin{thmA}\label{thm:smooth}
Let $X$ be a nonempty, closed and definable subset of $\mathbb{R}^n$.
If the $C^k$-smooth points of $X$ are not dense in $X$ for some $k \geq 0$, then $\mathcal{R}$ defines an isomorphic copy of $\mathcal{B}$.
If the affine points of $X$ are not dense in $X$, then $\mathcal{R}$ either defines an isomorphic copy of $\mathcal{B}$ or is field-type.
\end{thmA}

\noindent When $\mathcal{R}$ is o-minimal\!
\footnote{Recall that $\mathcal{R}$ is \textbf{o-minimal} if every nonempty definable subset of $\mathbb{R}$ is a finite union of open intervals and singletons, and that an o-minimal structure cannot define an isomorphic copy of $(\N,+1)$ by \cite[Remark 2.14]{Lou}.}, the first statement of Theorem A can be deduced from the o-minimal cell decomposition theorem and a theorem of Laskowski and Steinhorn \cite[Theorem 3.2] {LasStein}, stating that definable functions in such expansions are $C^k$ outside a definable lower-dimensional set. The second statement of Theorem A follows in the o-minimal setting from work by Marker, Peterzil and Pillay \cite{MPP}, who essentially show that an o-minimal expansion of $(\R,<,+)$ that defines a nowhere locally affine set, is field-type. While not necessary for the proof of Theorem A, we will show in Section \ref{section:ckdef} that for every expansion $\mathcal{R}$ the $C^k$-smooth points of a definable set are definable again.\newline

\noindent In his original work, Mandelbrot called a set $X$ a fractal if the topological dimension of $X$ is strictly less than the Hausdorff dimension of $X$. Topological dimension here refers to either small inductive dimension, large inductive dimension, or Lebesgue covering dimension. On subsets of $\R^n$ these three dimensions coincide (see Engelking \cite{Engelking} for details and definitions).
Given Mandelbrot's definition, it is natural to explore situations in which metric dimensions and topological dimensions do not coincide.
Here, we will discuss three important and well-known metric dimensions: Hausdorff, packing, and Assouad dimension.
We refer the reader to Heinonen~\cite{Juha}, Mattila~\cite{Mattila}, or Fraser~\cite{Fraser2020} for the definitions of these dimensions and the basic facts we apply.
It is well-known that
$$ \dim X \leq \dim_{\operatorname{Hausdorff}} X \leq \dim_{\operatorname{Packing}} X \leq \dim_{\operatorname{Assouad}} X$$
for all nonempty subsets $X$ of $\R^n$.
Here and below $\dim X$ is the topological dimension of $X$.
Commonly used metric dimensions are bounded below by the  topological dimension and above by the Assouad dimension.\newline

\noindent In this paper, we will obtain the following theorem as a corollary of Theorem A.

\begin{thmB} Let $X$ be a nonempty, closed and definable subset of $\mathbb{R}^n$. Then
\begin{itemize}
\item[(i)] If $X$ is bounded, $\dim X < \dim_{\operatorname{Assouad}} X$, and the affine points of $X$ are not dense in $X$, then $\mathcal{R}$ defines an isomorphic copy of $\mathcal{B}$.
\item[(ii)] If $X$ does not have affine points and $\dim X < \dim_{\operatorname{Packing}} X$, then $\mathcal{R}$ defines an isomorphic copy of $\mathcal{B}$.
\end{itemize}
\end{thmB}

\noindent Statement (i) in Theorem B does not generalize to unbounded sets, and in statement (ii) packing dimension can not be replaced by Assouad dimension.
The structure $(\R,<,+,\sin)$ is known to be NIP, and hence does not define an isomorphic copy of $\mathcal{B}$.
However, the reader can check that the $(\R,<,+,\sin)$-definable set
\[
\{ (x, t + \sin(x) ) : t \in \pi\mathbb{Z}, x \in \R \}
\]
has Assouad dimension two, topological dimension one, and no affine points.
We discuss $(\R,<,+,\sin)$ and related structures in Section~\ref{section:top-assouad} below.
\newline

\noindent Theorem B is not the first result that establishes model-theoretic wildness in expansions of the real line in which metric dimensions do not coincide with the topological dimension. For $a\in \R$, let $\lambda_a : \R \to \R$ map $x$ to $ax$. We denote by $\mathbb{R}_{\mathrm{Vec}}$ the ordered vector space $(\R,<,+,(\lambda_a)_{a\in \R})$.

\begin{fact}[{Fornasiero, Hieronymi and Walsberg \cite[Theorem A]{FHW-Compact}}]\label{thm:equal-dim}
Suppose $\mathcal{R}$ expands $\mathbb{R}_{\mathrm{Vec}}$. Let $X \subseteq \mathbb{R}^n$ be nonempty, closed and definable.
If the topological dimension of $X$ is strictly less than the Hausdorff dimension of $X$, then $\mathcal{R}$ defines every bounded Borel subset of every $\mathbb{R}^n$.
\end{fact}

\noindent Observe that whenever $\mathcal{R}$ defines every bounded Borel subset of every $\mathbb{R}^n$, it also defines an isomorphic copy of $\mathcal{B}$. Thus Theorems A and B can be seen as an analogue of Fact \ref{thm:equal-dim} when $\mathcal{R}$ does not necessarily expand $\mathbb{R}_{\mathrm{Vec}}$. Note that there exists a compact subset $X$ of $\R$ with topological dimension zero and positive packing dimension such that $(\R_{\text{Vec}},X)$ does not define all bounded Borel sets (see \cite[Section 7.2]{FHW-Compact}).
There are even stronger results for expansions of the real field.

\begin{fact}[{Hieronymi, Miller \cite[Theorem A]{HM}}]\label{fact:equal-dim1}
Suppose $\mathcal{R}$ expands $(\R,+,\cdot)$ and $X \subseteq \mathbb{R}^n$ is nonempty, closed and definable.
If the topological dimension of $X$ is strictly less than the Assouad dimension of $X$, then $\mathcal{R}$ defines every Borel subset of every $\mathbb{R}^n$.
\end{fact}


\noindent Theorems A and B arguably show that if $\mathcal{R}$ defines an object that can be called a fractal, then $\mathcal{R}$ defines an isomorphic copy of $\mathcal{B}$. Hence any tameness condition that is preserved by interpretability (such as NIP and NTP2), fails for such $\mathcal{R}$ whenever it fails for $\mathcal{B}$. It is natural to wonder whether such expansions fail every model-theoretic tameness condition that fails for $\mathcal{B}$. It is known (see the next paragraph) that there exist expansions of $(\R,<,+)$ that define fractal subsets of $\R^n$ and are bi-interpretable with $\mathcal{B}$. It can be argued that in order to be consider \emph{model-theoretic}, a tameness condition should  be preserved unter bi-interpretability\footnote{Indeed, according to Pillay \cite{BKPS} ``the business of `pure' model theory becomes the classification of first order theories up to bi-interpretability.''}. Thus every such tameness condition satisfied by \emph{all} expansions of $(\R,<,+)$ that define fractal objects in sense of Theorems A and B, has to be satisfied by $\mathcal{B}$ as well. Hence our theorems are optimal in this sense.\newline

\noindent  We describe an example of an expansion $\mathcal{R}$ that is bi-interpretable with $\mathcal{B}$ and defines fractal subsets of $\R^n$. Fix a natural number $r \geq 2$.
Let $V_r(x,u,d)$ be the ternary predicate on $\R$ that holds whenever $u = r^n$ for $n \in \mathbb{N}_{\geq 1}$ and there is a base $r$ representation of $x$ with $n$th digit $d$. Let $\sigma_r : r^{\N} \to r^{-\N}$ be the function that maps $r^{n}$ to $r^{-n}$ for all $n\in \N$. We let $\mathcal{R}_r$ be $(\R,<,+,V_r,\sigma_r)$.
It is easy to see that $\mathcal{R}_3$ defines the middle-thirds Cantor set, the Sierpinski triangle, and the Menger carpet. Adjusting the work in Boigelot, Rassart and Wolper \cite{BRW} to account for $\sigma_r$, one can easily show that $\mathcal{B}$ and $\mathcal{R}_r$ are bi-interpretable.\newline


%

\noindent
\noindent We finish with a few open questions.
We do not know if Theorem A remains true when ``$C^k$" is replaced with ``$C^\infty$".
Note that by Rolin, Speissegger and Wilkie \cite{RSW}, there is an o-minimal expansion of $(\R,+,\cdot)$ that defines a function $f : \R \to \R$ that is not $C^\infty$ on a dense, definable and open subset of $\R$.
However, this function is still $C^\infty$ on a dense open subset of $\R$.
These considerations lead to the following question:

\begin{Q}
Is there an o-minimal expansion of $(\R,+,\cdot)$ that defines a function $f : \R \to \R$ that is not $C^\infty$ on a dense open subset of $\R$?
\end{Q}

\noindent The author of \cite{LGal} indicated to us that the it might be possible to adapt ideas from that paper to construct such an expansion.

\begin{Qs}\label{q:1}
Let $X$ be a nonempty, closed and definable subset of $\R^n$.
If the topological dimension of $X$ is strictly less than the Hausdorff dimension of $X$, then must $\mathcal{R}$ define an isomorphic copy of $\mathcal{B}$?
\end{Qs}

\noindent  Observe that Theorem B gives an affirmative answer to Question~\ref{q:1} under the additional assumption that $X$ does not have affine points. We do not even know the answer to the following weaker question.

\begin{Qs}\label{q:2}
If $\mathcal{R}$ defines an uncountable nowhere dense subset of $\R$, must $\mathcal{R}$ define an isomorphic copy of $\mathcal{B}$?
Weaker: if $\mathcal{R}$ defines an uncountable nowhere dense subset of $\mathbb{R}$, must $\mathcal{R}$ have $\mathrm{IP}$?
\end{Qs}

\subsection*{Acknowledgments} The first author was partially supported by NSF grant DMS-1654725 and the \emph{Hausdorff Center for Mathematics} at the University of Bonn.  The authors thank Chris Miller for helpful feedback on an earlier version of this paper.

\section{Conventions, notation and background}
\subsection{Conventions and notations.} Throughout $m,n$ are natural numbers, $i,j,k,l$ are integers and $s,t,\delta,\varepsilon$ are real numbers.
Throughout ``dimension" is topological dimension unless stated otherwise.\newline

\noindent Let $X$ be a subset of $\R^n$.
Let $\dim X$ be the  topological dimension of $X$, let $\dim_{\operatorname{Hausdorff}} X$ be the Hausdorff dimension of $X$, let $\dim_{\operatorname{Packing}} X$ be the packing dimension of $X$ and let  $\dim_{\operatorname{Assouad}} X$ be the Assouad dimension of $X$. Furthermore, $\cl(X)$ and $\Int(X)$ are the closure and interior of $X$, and $\text{Bd}(X) := \cl(X) \setminus \Int(X)$ is the boundary of $X$, all with respect to the usual order topology.\newline
For $A \subseteq \mathbb{R}^{m + n}$ and $x \in \mathbb{R}^m$, set
\[
A_x := \{ y \in \mathbb{R}^n : (x,y) \in A \} .
\]
Let $f: X \to Y$ be a function. We write $\Gamma(f)$ for the graph of a function $f$, and $f|_Z$ for the restriction of $f$ to a subset $Z\subseteq X$.\newline

\noindent  A family $( A_t )_{t > 0}$ of sets is \textbf{increasing} if $s < t$ implies $A_s \subseteq A_t$, and \textbf{decreasing} if $s < t$ implies $A_t \subseteq A_s$.\newline

\noindent Throughout $\|\cdot\|$ is the $\ell_\infty$-norm and an ``open ball" is an open $\ell_\infty$-ball. For $x\in \R^n$, we denote by $B_{\varepsilon}(x)$ the open ball of radius $\varepsilon$ around $x$. We use the $\ell_\infty$-norm as opposed to the $\ell_2$-norm, since $\ell_\infty$ is $(\R,<,+)$-definable.
 All dimensions of interest are bi-lipschitz invariants and therefore unaffected by our choice of norm.

\subsection{Background}\label{section:background}
We review definitions and results from the theory of first-order expansions of $(\R,<,+)$.
An \textbf{$\omega$-orderable set} is a definable set that is either finite or admits a definable order of order-type $\omega$.
One should think of ``$\omega$-orderable sets" as ``definably countable sets".
A \textbf{dense $\omega$-order} is an $\omega$-orderable subset of $\mathbb{R}$ that is dense in some nonempty open interval.
 We say $\mathcal{R}$ is \textbf{type A} if it does not admit a dense $\omega$-order, \textbf{type C} if it defines every bounded Borel subset of every $\mathbb{R}^n$, and \textbf{type B} if it is neither type A nor type C.
 It is easy to see that these three classes of structures are mutually exclusive. We refer the reader to \cite{HW-continuous} for a more detailed discussion of this trichotomy and its relevance. For the reader's convenience, we collect several of the main results from \cite{HW-Monadic, HW-continuous}.
 
\begin{fact}\label{thm:old}
Let $k\geq 1$, let $U \subseteq \R^m$ be a definable open set, and let $f : U \to \mathbb{R}^n$ be definable and continuous.
\begin{enumerate}
\item If $\mathcal{R}$ is not type A, then $\mathcal{R}$ defines an isomorphic copy of $\mathcal{B}$ (see \cite[Theorem A]{HW-Monadic}).
\item If $\mathcal{R}$ is type B, then $\mathcal{R}$ is not field-type (see \cite[Theorem C]{HW-continuous}).
\item If $\mathcal{R}$ is type A, then there is a dense, open and definable subset $V$ of $U$ on which $f$ is $C^k$ (see \cite[Theorem B]{HW-continuous}).
\item If $\mathcal{R}$ is type A and not field-type, then there is a dense, open and definable subset $V$ of $U$ on which $f$ is locally affine (see \cite[Theorem A]{HW-continuous}).
\item If $\mathcal{R}$ is not field-type, $U$ is connected and $f$ is $C^2$, then $f$ is affine (see \cite[Theorem 8.1]{HW-continuous}). 
\end{enumerate}
\end{fact}

\noindent A set $X\subseteq \R^n$ is $\textbf{D}_{\mathbf{\Sigma}}$ if $X = \bigcup_{s,t > 0} X_{s,t}$ for a definable family $( X_{s,t} )_{s,t > 0}$ of compact subsets of $\mathbb{R}^n$ such that $X_{s,t} \subseteq X_{s,t'}$ when $t \leq t'$ and $X_{s',t} \subseteq X_{s,t}$ when $s \leq s'$.
We say that such a family \textbf{witnesses} that $X$ is $\Dsig$.
Note that a $\DSig$ set is definable and that every $\DSig$ set is $F_\sigma$.

\begin{fact}[{Dolich, Miller and Steinhorn \cite[1.10]{DMS1}}]\label{prop:image}
Open and closed definable sets are $\Dsig$, a finite union or finite intersection of $\Dsig$ sets is $\Dsig$, and the image of a $\Dsig$ set under a continuous definable function is $\Dsig$.
\end{fact}

\noindent A key result about $\DSig$ sets in type A structures is the following Strong Baire Category Theorem, or \textbf{SBCT}.

\begin{fact}[{\cite[Theorem 4.1]{FHW-Compact}}]\label{SBCT}
Suppose $\mathcal{R}$ is type A. Let $X\subseteq \R^n$ be a $\Dsig$ set witnessed by the definable family $(X_{s,t})_{s,t > 0}$.
Then
\begin{enumerate}
    \item $X$ either has interior or is nowhere dense,
    \item  if $X$ is dense in a definable open set $U$, then the interior of $X$ is dense in $U$,
    \item if $X$ has interior, then $X_{s,t}$ has interior for some $s,t > 0$,
    \item  if $(X_t)_{t > 0}$ is an increasing family of $\DSig$ sets and $\bigcup_{t > 0} X_t$ has interior, then $X_t$ has interior for some $t > 0$.
\end{enumerate}
\end{fact}

\noindent The latter two claims follow by applying the Baire Category Theorem to the first claim.
The following corollary follows from SBCT and the fact that the closure and the interior of a $\DSig$ set are also $\DSig$.

\begin{Cor}\label{cor:bd} Suppose $\mathcal{R}$ is type $A$.
Let $X\subseteq \R^n$ be $\DSig$. Then $\mathrm{Bd}(X)$ is nowhere dense.
\end{Cor}

\noindent We also need the following $\mathbf{\Dsig}$-\textbf{selection} result in type A structures.

\begin{fact}[{\cite[Proposition 5.5]{FHW-Compact}}]\label{select}
Suppose $\mathcal{R}$ is type A. Let $X \subseteq \mathbb{R}^{m+n}$ be $\Dsig$, and let $U \subseteq \mathbb{R}^m$ be a nonempty, open set contained in the coordinate projection of $X$ onto $\mathbb{R}^m$.
Then there is a nonempty, open and definable set $V \subseteq U$ and a continuous and definable function $f : V \to \R^n$ such that the graph $\Gamma(f)$ is contained in $X$.
\end{fact}

\noindent We collected two special cases of a general result on additivity of dimension~\cite[Theorem E]{FHW-Compact}.

\begin{fact}\label{prop:project}
Suppose $\mathcal{R}$ is type A. Let $d \in \N$ and let $A \subseteq \mathbb{R}^n$ be a $\Dsig$ set such that $\dim A = d$ for $d\in \N$ with $1\leq d\leq n-1$. Then there is a coordinate projection $\pi : \mathbb{R}^n \to \mathbb{R}^d$ and a nonempty open $U \subseteq \mathbb{R}^d$ contained in $\pi(A)$ such that for all $x\in U$
\[
\dim (\pi^{-1}(\{x\}))=0.
\]
\end{fact}

\begin{fact}\label{prop:project1}
Suppose $\mathcal{R}$ is type A. Let $X \subseteq \R^n$ be $\DSig$, and let $f : X \to \R^m$ be a continuous and definable function.
Then $\dim f(X) \leq \dim X$.
\end{fact}

\noindent One corollary of Fact~\ref{prop:project1} is that type A expansions cannot define space-filling curves \cite[Theorem E]{FHW-Compact}. For our purposes a \textbf{Cantor set} is a nonempty compact nowhere dense subset of $\R$ without isolated points.

\begin{fact}[{\cite[Theorem B]{HW-Monadic}}]\label{prop:cantor0}
If $\mathcal{R}$ defines a Cantor set, then $\mathcal{R} $ defines an isomorphic copy of $\mathcal{B}$.
\end{fact}

\noindent If $X \subseteq \R$, then $p \in X$ is $C^k$-smooth (for any $k \geq 0$) if and only if $p$ is either isolated in $X$ or lies in the interior of $X$.
Therefore Fact~\ref{prop:cantor0} yields Theorem A for definable subsets of $\R$.

\section{Hausdorff continuity of definable families}
\noindent
Throughout this section $\mathcal{R}$ is assumed to be type A
and $U$ is assumed to be a fixed, nonempty, open and definable subset of $\mathbb{R}^m$. The goal of this section is to show that definable families of $\DSig$ subsets of $\R^n$ indexed by $U$ are Hausdorff continuous on a dense, open subset of $U$. We make this statement precise in Proposition \ref{lem:hd}. Let $\pi:\R^m \times \R^n \to \R^m$  be the coordinate projection onto the first $m$ coordinates. \newline

\noindent We first recall some useful notions from metric geometry.
Let $f : (X,d_X) \to (Y,d_Y)$ be a function between metric spaces.
 The \textbf{oscillation} of $f$ at $x \in X$ is the supremum of all $\delta \geq 0$ such that for every $\varepsilon > 0$ there are $y,z \in X$ such that $d_X(x,y) <\varepsilon$, $ d_X(x,z) < \varepsilon$ and $d_Y\left(f(y), f(z)\right) > \delta$.
 Recall that $f$ is continuous at $x$ if and only if the oscillation of $f$ at $x$ is zero. Furthermore, the set of $x \in X$ at which the oscillation of $f$ is at least $\varepsilon$ is closed for every $\varepsilon > 0$.

\begin{Def}
For $\delta \geq 0$ and $X\subseteq \R^n$, we define
\[
B_{\delta}(X) := \bigcup_{x\in X} \{ y \in \R^n \ : \ \| x - y \| < \delta\}.
\]
The \textbf{Hausdorff distance} $d_{\mathscr{H}}(X,Y)$ between two nonempty subsets $X$ and $Y$ of $\R^n$ is
\[
d_{\mathscr{H}}(X,Y) := \inf \{ \delta \geq 0 \ : \ X \subseteq B_{\delta}(Y), Y \subseteq B_{\delta}(X)\};
\]
that is, the infimum of the set of all $\delta > 0$ such that for every $x \in X$ there is a $y \in Y$ such that $\| x - y \| < \delta$ and for every $y \in Y$ there is an $x \in X$ such that $\| x - y \| < \delta$.
\end{Def}
 
\noindent The Hausdorff distance between a bounded subset of $\R^n$ and its closure is zero.
 The Hausdorff distance restricts to a separable complete metric on the collection $\mathcal{C}$ of all nonempty compact subsets of $\mathbb{R}^m$. The following observation is an immediate consequence of the definition of $d_{\mathscr{H}}$.

 \begin{Lem}\label{lem:hd0}
 Let $W$ be a bounded, open subset of $\R^n$, let $X,Y\subseteq W$ and let $\mathcal{D}$ be a collection of open balls of diameter at least $\varepsilon$ covering $W$  such that
 \[
  \{ B \in \mathcal{D} : B \cap X \neq \emptyset \} = \{ B \in \mathcal{D} : B \cap Y \neq \emptyset \}.
  \]
 Then $d_{\mathscr{H}}(X,Y) \leq \varepsilon$.
 \end{Lem}

\begin{Def} Let $Z \subseteq \mathbb{R}^m$ and let $\Cal A = (A_x)_{x \in Z}$ be a family of nonempty bounded subsets of $\mathbb{R}^n$. Let $M_{\Cal A}: Z \to \mathcal{C}$ map $x\in Z$ to  $\cl(A_x)$.
We say that $\Cal A$ is \textbf{HD-continuous} if $M_{\Cal A}$ is continuous.    
\end{Def}

\noindent For $\varepsilon >0$, let $\mathcal{O}_\varepsilon(\Cal A)$ be the set of points in $Z$ at which $M_{\Cal A}$ has oscillation at least $\varepsilon$. Let $\mathcal{O}(\Cal A)$ be the set of points at which $M_{\Cal A}$ has positive oscillation. Observe that the complement of $\mathcal{O}(\Cal A)$ is the set of points at which $M_{\Cal A}$ is continuous, and that each $\mathcal{O}_\varepsilon(\Cal A)$ is closed in $Z$.\newline

\noindent We say that $A \subseteq \R^{m +n }$ is \textbf{vertically bounded} if there is an open ball $W\subseteq\R^n$ such that $A_x \subseteq W$ for all $x \in \R^m$.


\begin{Prop}\label{lem:hd}
Let $A \subseteq \R^{m} \times \mathbb{R}^{n}$ be $\Dsig$ and vertically bounded such that $U\subseteq \pi(A)$.
Then there is a dense definable open subset $V$ of $U$ such that $(A_x)_{x \in V}$ is $\mathrm{HD}$-continuous.
\end{Prop}

\begin{proof}
Let $\Cal A$ be the definable family $(A_x)_{x \in U}$.
We show that $\mathcal{O}(\Cal A)$ is nowhere dense and take $V$ to be the interior of the complement of $\mathcal{O}(\Cal A)$ in $U$. Note that then $V$ is definable, since $\mathcal{O}(\Cal A)$ is. It suffices to establish that every point in $U$ has a neighbourhood $U'$ such that $\mathcal{O}(\mathcal{A'})$ is nowhere dense in $U'$ where $\Cal A'$ is the restricted family $(A_x)_{x \in U'}$. Therefore, we may assume without loss of generality that $U$ is an open ball and in particular, connected. 
Since $(\mathcal{O}_\varepsilon(\mathcal{A}) )_{\varepsilon > 0}$ witnesses that $\mathcal{O}(\Cal A)$ is $\Dsig$, it suffices to show that $\mathcal{O}_\varepsilon(\Cal A)$ is nowhere dense for all $\varepsilon > 0$, and then apply SBCT to obtain nowhere density of $\Cal O(\Cal A)$.\newline

\noindent  Fix $\varepsilon > 0$.
 Let $W$ be an open ball in $\R^n$ such that $A_x \subseteq W$ for all $x \in U$.
 Let $\mathcal{D}$ be a finite collection of closed balls of diameter $\leq \varepsilon$ covering $W$. For each $B\in \mathcal{D}$,
 set
\[
E_B := \pi([\mathbb{R}^m \times B] \cap A);
\]
 that is, $E_B$ is the set of $x \in U$ such that $A_x$ intersects $B$. By
Fact~\ref{prop:image} each $E_B$ is $\DSig$, and hence $\text{Bd}(E_B)$ is nowhere dense in $U$ for each $B \in \mathcal{D}$ by Corollary~\ref{cor:bd}.
Recall that the boundary of a subset of a topological space is always closed.
Therefore
$$Y := \bigcap_{B \in \mathcal{D}} U \setminus \text{Bd}(E_B) $$
is dense and open in $U$.
We show that $Y$ is a subset of the complement of $\mathcal{O}_\varepsilon(\mathcal{A})$. The claim that $\mathcal{O}_\varepsilon(\mathcal{A})$ is nowhere dense, follows.\newline

\noindent Let $p \in Y$. We prove that the oscillation of $M_{\Cal A}$ at $p$ is at most $\varepsilon$.
Let $R$ be an open ball with center $p$ contained in $Y$. Observe that for all $B \in \mathcal{D}$ 
\[
R \cap \text{Bd}(E_B) = \emptyset.
\] 
Since $R$ and $U$ are both connected, we have that for every $B\in \mathcal{D}$, either
\[
R\subseteq E_B \text{ or } R\cap E_B= \emptyset.
\]
Let $q \in R$. For every $B \in \mathcal{D}$, we have that $q \in E_B$ if and only if $p \in E_B$. 
That is, for every $B\in \mathcal{B}$, 
\[
A_q \cap B \neq \emptyset \text{ if and only if } A_p \cap B \neq \emptyset.
\]
By
Lemma~\ref{lem:hd0} we immediately get that $d_{\mathscr{H}}(A_p, A_q) \leq \varepsilon$.
\end{proof}

\noindent We say that a point $p$ in a subset $X$ of $\R^n$ is $\mathbf{\varepsilon}$-\textbf{isolated} if $\|p - q \| \geq \varepsilon$ for all $q \in X$ with $p\neq q$.
We need the following lemma from metric geometry, whose proof we include for the convenience of the reader.

\begin{Lem}\label{lem:hdl}
Let $\varepsilon >0$, and let $A\subseteq \R^{m} \times \R^{n}$ be vertically bounded such that $U\subseteq \pi(A)$ and $( A_x )_{ x \in U }$ is HD-continuous.
Then the set of $(x,y) \in A$ such that $y$ is $\varepsilon$-isolated in $A_x$, is closed in $A$.
\end{Lem}

\begin{proof}
Let $S_{\varepsilon}\subseteq A$ be the set of $(x,y) \in A$ such that $y$ is $\varepsilon$-isolated in $A_x$.
We need to show that $S_{\varepsilon}$ is closed in $A$. 
Let $(x,y)\in A$ and let $((x_i,y_i))_{i \in \N}$ be a sequence of elements in $S_{\varepsilon}$ converging to $(x,y)$.
Towards a contradiction, suppose that there is $z\in A_x$ such that $\|(x,z)-(x,y)\|< \varepsilon$.
Let $\gamma \in \R_{> 0}$ be such that $\|(x,z)-(x,y)\| < \varepsilon - 2 \gamma$. As $((x_i,y_i))_{i \in \mathbb{N}}$ converges to $(x,y)$,  there is $i_0 \in \N$ such that for all $i\in \N_{\geq i_0}$
\[
\|(x,y) - (x_i,y_i)\| < \gamma.
\]
Since $(A_x)_{x \in U}$ is HD-continuous, we may also assume that for all $i\in \N_{\geq i_0}$
\[
d_H(A_x ,A_{x_i}) < \gamma.
\]
Thus we can pick for each natural number $i \in \N_{\geq i_0}$ an element $z_i \in A_{x_i}$ such that
\[
 \| (x,y) - (x_i,z_i) \| < \gamma.
\]
Thus for all $i\in \N_{\geq i_0}$
\begin{align*}
|y_i - z_i| &= \| (x_i,y_i) - (x_i,z_i) \|\\
&\leq \|(x_i,y_i) - (x,y)\| + \|(x,y) - (x,z)\| + \| (x,z) -(x_i,z_i)\|\\
&< \gamma + (\varepsilon - 2\gamma) + \gamma< \varepsilon.
\end{align*}
This contradicts that $(x_i,y_i)\in S_{\varepsilon}$.
\end{proof}

\noindent
If $X \subseteq \R^m$ is a $D_\Sigma$ set and $Y \subseteq X$ is definable and closed in $X$, then $Y$ is $D_\Sigma$.
Corollary~\ref{cor:vert} follows.

\begin{Cor}\label{cor:vert}
Let $A\subseteq U \times \R^{n}$ be a vertically bounded $\Dsig$ set such that $(A_x)_{x \in U}$ is HD-continuous.
Then the set of $(x,y) \in A$ such that $y$ is $\varepsilon$-isolated in $A_x$, is $\DSig$.
\end{Cor}

\section{$C^k$-smooth points on $\Dsig$-sets}
\noindent
We prove Theorem A in this section. Recall that if $\mathcal{R}$ is not type A, then $\mathcal{R}$ defines an isomorphic copy of $\mathcal{B}$ by Fact~\ref{thm:old}(1). If $\mathcal{R}$ defines a Cantor subset of $\R$, then $\mathcal{R}$ defines an isomorphic copy of $\mathcal{B}$ by Fact~\ref{prop:cantor0}. Thus it suffices to show  Theorem A (and Theorem B) under the assumption that $\mathcal{R}$ is type A and does not define a Cantor subset of $\R$. We suppose throughout the remainder of this section that $\mathcal{R}$ is type A. Let $\pi : \R^{m+n} \to \R^m$ be the coordinate projection onto the first $m$ coordinates. \newline

%

\noindent We need another elementary lemma from real analysis, whose proof we again include for the reader's convenience.

\begin{Lem}\label{lem:tube}
Let  $\varepsilon > 0$, let $A \subseteq \mathbb{R}^{m + n}$, let $U \subseteq \mathbb{R}^m$ be nonempty and open, and let $f : U \to \mathbb{R}^n$ be continuous such that $f(x)$ is an $\varepsilon$-isolated element of $A_x$ for all $x \in U$.
Then there are nonempty and open sets $V \subseteq U$ and  $W \subseteq \mathbb{R}^n$ such that 
\[
A \cap (V \times W) = \Gamma(f|_V).
\]
\end{Lem}

\begin{proof}
Let $p \in U$. By continuity of $f$, there is an open neighborhood $V \subseteq U$ of $p$ such that for all $q\in V$
\[
\| f(q) - f(p) \| < \frac{1}{2}\varepsilon.
\]
Let $W$ be the the open ball with center $f(p)$ and radius $\frac{1}{2}\varepsilon$.
We now show that the conclusion of the lemma holds. The right-to-left inclusion follows immediately from the definition of $W$.
We prove the left-to-right inclusion.
Let $(x,y) \in A \cap (V \times W)$.
By definition of $V$ and $W$, 
\[
\| y - f(x) \| \leq \| y - f(p) \| +  \| f(p) - f(x) \| < \varepsilon.
\]
Since $f(x)$ is $\epsilon$-isolated in $A_x$ and $y \in A_x$,  we have $y = f(x)$.
\end{proof}


\begin{Lem}\label{lem:smoothpoint1}
Let $k \geq 0$, let $A \subseteq \mathbb{R}^{m + n}$ be $\DSig$, and let $U\subseteq \R^m$ be a nonempty, open and definable set such that $U\subseteq \pi(A)$ and the isolated points of $A_x$ are dense in $A_x$ for all $x \in U$.
Then there exist a nonempty, open and definable set $V \subseteq U$ and $W \subseteq \mathbb{R}^n$ and a definable $C^k$-function $f : V \to W$ such that
\[
A \cap (V \times W) = \Gamma(f).
\]
If $\mathcal{R}$ is not field-type, then we can take $f$ to be affine.
\end{Lem}
\begin{proof}
We first reduce to the case when $A$ is vertically bounded. For $r\in \R_{>0}$, set
\[
 A(r) := \{ (x,y) \in A : \| y \| < r \}.
\]
Thus $\bigcup_{r > 0} A(r) = A$.
Then $( \pi \left( A(r) \right) )_{r > 0}$ is an increasing definable family of $\DSig$ sets such that
\[
U\subseteq \bigcup_{r > 0} \pi(A(r)).
\]
By Fact \ref{SBCT}, there is $t\in \R_{>0}$ such that $\pi(A(t))$ has interior. After replacing $U$ with $\Int(A(t))$ and $A$ with $A(t)$ if necessary, we may assume that $A$ is vertically bounded. After applying Lemma~\ref{lem:hd} and replacing $U$ with a smaller nonempty, open and definable set, we may assume that $( A_x )_{ x \in U}$ is HD-continuous.
For each $\varepsilon > 0$ set 
\[
S_\varepsilon := \{(x,y) \in A \ : \  y  \text{ is $\varepsilon$-isolated in $A_x$}\}.
\]
By Corollary~\ref{cor:vert} each $S_\varepsilon$ is $\Dsig$. Because $A_x$ has an isolated point for each $x \in U$, we get that 
\[
U \subseteq \bigcup_{\varepsilon  > 0} \pi(S_\varepsilon).
\]
Since $( \pi(S_\epsilon) )_{ \epsilon > 0 }$ is an increasing family of $\DSig$ sets, Fact \ref{SBCT} gives a $\delta>0$ such that  $\pi(S_\delta)$ has interior in $U$.
After replacing $U$ with a smaller nonempty definable open set if necessary, we may assume that $U$ is contained in $\pi(S_\delta)$.
Applying $\DSig$-selection we obtain a nonempty, open and definable set $V \subseteq U$ and a continuous and definable function $f : V \to \mathbb{R}^n$ such that $(x,f(x)) \in S_\delta$ for all $x \in V$. Thus $f(x)$ is $\delta$-isolated in $A_x$ for all $x \in V$.
After applying Fact~\ref{thm:old}(3) and shrinking $V$ if necessary, we may assume that $f$ is $C^k$.
Now apply Lemma~\ref{lem:tube}.\newline

\noindent Suppose that $\mathcal{R}$ is not field-type. Then after applying Fact~\ref{thm:old}(4) and shrinking $V$ if necessary, we have that $f$ is affine on $V$.
\end{proof}

\begin{Lem}\label{lem:cantor1}
Suppose $\mathcal{R}$ does not define a Cantor set. Let $A \subseteq \mathbb{R}^n$ be $\DSig$. If $\dim A = 0$, then the isolated points of $A$ are dense in $A$.
\end{Lem}

\noindent It is easy to see that $\mathcal{R}$ defines a Cantor subset of $\R$ if and only if it defines a nonempty, nowhere dense subset of $\R$ without isolated points (take closures).

\begin{proof}[Proof of Lemma \ref{lem:cantor1}]
We proceed by induction on $n$.
Suppose $n=1$. Since $\dim A =0$, we know that $A$ is nowhere dense by SBCT. 
Let $U\subseteq \R$ be an open set that intersects $A$. Since nowhere dense, definable subsets of $\R$ have isolated points, the intersection $A\cap U$ must contain an isolated point. Thus the isolated points of $A$ are dense in $A$.\newline

\noindent Now suppose $n>1$.
Let $\pi_0 : \R^n \to \R^{n - 1}$ be the coordinate projection onto the first $n$ coordinates.
By Fact~\ref{prop:project1} we have that $\dim \pi_0(A) = 0$.
Since $\pi_0(A)$ is $\DSig$, it contains an isolated point $x$ by induction.
Since $A_x$ is a definable subset of $\R$ with topological dimension 0, it also contains an isolated point $t$.
This $(x,t)$ is an isolated point in $A$.
\end{proof}

\noindent We are now ready to prove Theorem A. In the proof below we apply the fact that an arbitrary subset of $\R^n$ has topological dimension $n$ if and only if it has nonempty interior (see \cite[Theorem 1.8.10]{Engelking}).

\begin{proof}[Proof of Theorem A]
As pointed out above, by Fact \ref{thm:old}(1) and Fact \ref{prop:cantor0} we can reduce to the case that $\mathcal{R}$ is type A and does not define a Cantor set. 
 Let $A \subseteq \mathbb{R}^n$ be $\DSig$ and nonempty. We prove that the $C^k$-smooth points of $A$ are dense in $A$. It is enough to show that every open and definable subset of $\R^n$ that intersects $A$, contains a $C^k$-smooth point.
Let $U \subseteq \mathbb{R}^n$ be an open and definable set that intersects $A$. Let $d$ be the topological dimension of $U \cap A$.
If $d = 0$, apply Lemma~\ref{lem:cantor1}.
If $d = n$, then $U \cap A$ has interior. Every interior point of $A$ is $C^k$-smooth. Now suppose that $1 \leq d \leq n - 1$.
By Fact~\ref{prop:project} there is a coordinate projection $\pi_1 : \mathbb{R}^n \to \mathbb{R}^d$ and a nonempty, open and definable set $V \subseteq \mathbb{R}^d$ such that $V \subseteq \pi_1(A)$ and $\dim(\pi_1^{-1}(\{x\})) = 0$ for all $x \in V$.
Without loss of generality we can assume that $\pi_1$ is the coordinate projection onto the first $d$ coordinates. Since $\dim(\pi_1^{-1}(\{x\})
)=0$ for all $x\in V$, we have that $\dim A_x =0$ for all $x\in V$. By Lemma~\ref{lem:cantor1}, the isolated points of $A_x$ are dense in $A_x$ for all $x \in V$.
Now apply Lemma~\ref{lem:smoothpoint1}.
\end{proof}


\section{Coincidence of dimensions and Theorem B}\label{section:dimensions}
\noindent
In this section, we prove Theorem B, which we recall as a convenience for the reader.

\begin{thmB} Let $X\subseteq \R^n$ be nonempty, closed  and definable. Then
\begin{itemize}
\item[(i)] If $X$ is bounded, the affine points of $X$ are not dense in $X$, and we have $\dim X < \dim_{\operatorname{Assouad}} X$, then $\mathcal{R}$ defines an isomorphic copy of $\mathcal{B}$.
\item[(ii)] If $X$ does not have affine points and $\dim X < \dim_{\operatorname{Packing}} X$, then $\mathcal{R}$ defines an isomorphic copy of $\mathcal{B}$.
\end{itemize}
\end{thmB}

\noindent We obtain Theorem B as a corollary of Theorem A and an extension of Fact \ref{fact:equal-dim1}. Recall that by Fact \ref{fact:equal-dim1} on
every bounded $\DSig$ set $X$ in a type A expansion of $(\R,+,\cdot)$, the Assouad dimension of $X$ equals the topological dimension of $X$. We refer the reader to \cite{HM} and Luukkainen \cite{Luukkainen} for more on the relevance and significance of this phenomenon. We now extend this result to type A expansions that are field-type.

\begin{Thm}
\label{dimconfiety}\footnote{This theorem first appeared in a preprint version of \cite{HW-continuous}, but has been removed from the final version of \cite{HW-continuous}.}
Suppose $\Cal R$ is type A and field-type. Let $X$ be a bounded $\DSig$ set.
Then the Assouad dimension of $X$ is equal to the topological dimension of $X$.
\end{Thm}

\noindent The coincidence of dimensions for $\DSig$ sets does not extend to type A expansions of $(\R,<,+)$.
For example, $S = \{ 1/n : n \in \N, n \geq 1 \}$ has Assouad dimension $1$, topological dimension $0$, and $( \R , < , + , ( x \mapsto \lambda x)_{ \lambda \in \R}, S ) $ is type A by \cite[Theorem B]{FHW-Compact}. The assumption that $X$ is bounded is also necessary.
The structure $( \R, <, +, \sin)$ is locally o-minimal and hence type A, see Section~\ref{section:top-assouad} below.
This structure is also field-type, as the sine function is $C^2$ and non-affine, and defines $\pi\Z = \sin^{-1}(\{0\})$. The latter set has topological dimension $0$ and Assouad dimension $1$.\newline

\noindent We need to clarify one notation before giving the proof of Theorem \ref{dimconfiety}. Let $e$ be the Euclidean metric on $\R^n$ and let $X\subseteq \R^n$. When we refer to the Assouad dimension of $X$, we mean the Assouad dimension of the metric space $(X,e_X)$, where $e_X$ is the restriction of the metric $e$ to $X$.

\begin{proof}[Proof of Theorem \ref{dimconfiety}]
Since $\Cal R$ is field-type, we know that $\Cal R$ defines a $C^2$ function $f : [0,1] \to \R$ with nonconstant derivative by \cite[Theorem A(4), Lemma 5.1]{HW-continuous} and Fact \ref{thm:old}(3).
By \cite[Lemma 6.2]{HW-continuous} there is an interval $I\subseteq \R$, continuous and definable functions $\oplus,\otimes : I^2 \to I$, an isomorphism $\tau :(I,<,\oplus,\otimes)\to (\mathbb{R},<,+,\cdot)$ and a subinterval $J\subseteq I$ such that the restriction of $\tau$ to $J$ is a $C^k$-diffeomorphism for $k\geq 1$.
It is clear from the proof of \cite[Lemma 6.2]{HW-continuous} that $\tau$ can be chosen to agree with $f'$ on $J$. Since $\tau$ is strictly increasing, we also have that $f'$ is strictly increasing.
Thus $f''$ is positive on $J$.
Since $f''$ is continuous,  we may suppose (after further shrinking $J$ if necessary) that $f''$ is bounded above and bounded away from zero on $J$.
It follows from the mean value theorem that $f'$ is bi-Lipschitz on $J$.\newline

\noindent Let $d$ be the natural metric on $(I,<,\oplus,\otimes)$ given by
\[
d(x,y) = \tau (\| x \ominus y \|_I ) \quad \text{for all } x,y \in I,
\]
where $\ominus$ and $\|\cdot\|_I$ are the inverse images of $-$ and $\|\cdot\|$ under $\tau$.
We now consider the metric space $(J,d_J)$ where $d_J$ is the restriction of $d$ to $J$. As $\tau$ is an ordered field isomorphism that agrees with $f'$ on $J$, we have
\[
d(x,y) = \| f'(x) - f'(y) \| \quad \text{for all } x,y \in J.
\]
Then the identity map $(J, e_J) \to (J,d_J)$ is a bi-Lipschitz equivalence, because $f'$ is bi-Lipschitz.
Let $d_n$ be the metric on $J^n$ given by
\[
d_n ( x, y ) = \max \{ d ( x_1,y_1) , \ldots, d(x_n,y_n) \}
\]
for all  $x = (x_1,\ldots,x_n), y = (y_1,\ldots,y_n)\in J^n. $
It is easy to see that the identity map $(J^n , d_n) \to (J^n, e_{J^n})$ is also a bi-Lipschitz equivalence.\newline

\noindent
Let $X\subseteq \R^n $ be a bounded $D_\Sigma$ set. Let $q \in \mathbb{Q}_{>0}, t \in \mathbb{R}^n$ be such that $qX + t$ is a subset of $J^n$.
Then $qX + t$ has the same Assouad dimension and topological dimension as $X$, because  invertible affine maps are bi-Lipschitz and bi-Lipschitz equivalences preserve both Assouad and topological dimension. After replacing $X$ with $qX + t$ if necessary we suppose that $X \subseteq J^n$. By \cite[Theorem E]{FHW-Compact} the topological dimension of the closure of $X$ agrees with the topological dimension of $X$. It follows directly from the definition of Assouad dimension that taking closures does not raise Assouad dimension of subsets of $\mathbb{R}^n$. It therefore suffices to prove the theorem for closed definable sets.\newline

\noindent We assume that $X$ is closed. Consider the structure $(I, < , \oplus, \otimes, X)$. This structure is isomorphic via $\tau$ to $\Cal S := (\R, < , +, \cdot, \tau(X))$.
Since $\Cal R$ is type A, so is $\Cal S$. In particular, $\Cal S$ cannot define every Borel set.
Therefore by Fact \ref{fact:equal-dim1} Assouad dimension and topological dimension agree on $(\tau(X),e_{\tau(X)})$. Since $\tau$ is bi-Lipschitz on $J$, it follows from the definition of $d_n$ that
the Assouad and topological dimensions of $(X, (d_n)_X)$ agree. Since $\operatorname{id} : (J^n, d_n) \to (J^n, e_{J^n})$ is bi-Lipschitz, the Assouad and topological dimensions of $X$ agree.
\end{proof}

\noindent
The reduction to the case when $X$ is closed is necessary as $(I,<,\oplus,\otimes,X)$ need not define a witness that $X$ is $\DSig$.
Before proving Theorem B, we prove the following corollary to Theorem~\ref{dimconfiety}.

\begin{Cor}\label{cor:harry}\footnote{We thank Harry Schmidt for asking about $(\R,<,+,\sin(1/t))$.}
Let $f : (0,1] \to \R$ be the function mapping $t$ to $\sin(1/t)$, and let $\tan: \R \setminus [\pi\Z + \frac{\pi}{2}] \to \R$ be the usual tangent function.
Then $(\R,<,+,f)$ and $(\R,<,+,\tan)$ are type C.
\end{Cor}

\noindent
The second claim of Corollary~\ref{cor:harry} requires the following result from metric geometry due to Garc\'{\i}a, Hare, and Mendivil~\cite[Proposition 3.1]{GHM} (see also Fraser and Yu \cite[Theorem 6.1]{FY}).

\begin{fact}
\label{fact:sequence}
Let $g : \R \to \R_{>0}$ be a differentiable function such that $g$ and $-g'$ are strictly decreasing on $[\delta,\infty)$ for some $\delta > 0$ and that both $g(t),-g'(t)$ tend to $0$ as $t \to \infty$.
Then one of the following holds:
\begin{enumerate}
\item $\{ g(n) : n \in \N\}$ has Assouad dimension one, or
\item there is $\alpha\in (0, 1)$ such that $g(n) \le \alpha^n$ for all sufficiently large $n\in \N$.
\end{enumerate}
In the second case $\{g(n) : n \in \N \}$ has Assouad dimenson zero.
\end{fact}

\begin{proof}[Proof of Corollary~\ref{cor:harry}]
Both structures are field-type as they define non-affine $C^2$ functions.
Observe that 
\[
\{ x \in (0,1] \ : \ f(t) = 0 \} =  \left\{ \frac{1}{\pi n}  \ : \ n \in \N_{>0}\right\}. 
\]
Since Assouad dimension is invariant under inversion~\cite[Theorem XIII.5.2]{Luukkainen} and rescaling, we obtain
\[
\dim_\mathrm{Assouad} \left\{ \frac{1}{\pi n}  \ : \ n \in \N_{>0}\right\} = \dim_\mathrm{Assouad} \pi\N = \dim_\mathrm{Assouad} \N = 1.
\]
By Theorem \ref{dimconfiety}, we know that $(\R,<,+,f)$ can not be type A. Thus $(\R,<,+,f)$ has to be type C by Fact \ref{thm:old}(2). \newline

\noindent Now let $\arctan: \R \to (-\pi/2,\pi/2)$ be the usual arctangent.
Note that $(\R,<,+,\tan)$ defines both $\arctan$ and $\pi\Z$.
Let $g : \R_{\ge 0} \to (0,\pi/2]$ be given by 
\[
g(t) = (\pi/2) - \arctan(\pi t).
\]
It is easy to see that $g$ satisfies the conditions of Fact~\ref{fact:sequence} and that $g(n)$ does not decrease exponentially.
Hence  $\{ g(n) : n \in \N\}$ has Assouad dimension one by Fact~\ref{fact:sequence}.
Note that $\{ g(n) : n \in \N \}$ is definable in $(\R,<,+,\tan)$. The conclusion that $(\R,<,+,\tan)$ is type C, follows as in the case of $(\R,<,+,f)$.
\end{proof}

\begin{proof}[Proof of Theorem B]
(i) Let $X \subseteq \mathbb{R}^n$ be nonempty, $\DSig$, and bounded such that the affine points of $X$ are not dense in $X$ and the topological dimension of $X$ is strictly less than the Assouad dimension of $X$. We have to show that $\mathcal{R}$ defines an isomorphic copy of $\mathcal{B}$. By Theorem A, either $\mathcal{R}$ defines an isomorphic copy of $\mathcal{B}$ or $\mathcal{R}$ is field-type. Suppose $\mathcal{R}$ is field-type. By Theorem~\ref{dimconfiety}, we get that $\mathcal{R}$ can not be of type A. However, every structure of type B or type C defines an isomorphic copy of $\mathcal{B}$.\newline

\noindent (ii) Let $X \subseteq \R^n$ be $\DSig$ such that $X$ has no affine points and the topological dimension of $X$ is strictly less than the packing dimension of $X$. We need establish that $\mathcal{R}$ defines an isomorphic copy of $\mathcal{B}$. It follows directly from the definition of the packing dimension that whenever $Y \subseteq \R^n$ is Borel, and $(Y_m)_{m \in \N}$ is a collection of Borel subsets of $Y$ covering $Y$, then $\dim_{\text{Packing}} Y = \sup_m \dim_{\text{Packing}} Y_m$.
The same statement holds for topological dimension provided each $Y_m$ is $F_\sigma$ by \cite[Corollary 1.5.4]{Engelking}.\newline

\noindent Given $m \in \Z^n$ we let $X_{m}$ be $([0,1]^n + m) \cap X $.
As each $X_{m}$ is $F_\sigma$, we have $\dim X = \sup_{m \in \Z^n} \dim X_{m}$.
Thus if $\dim X_{m} = \dim_{\text{Packing}} X_{m}$ for all $m \in \Z^n$, we obtain $\dim X = \dim_{\text{Packing}} X$.
Suppose that $\dim X_{m} < \dim_{\text{Packing}} X_{m}$ for some $m \in \Z^n$. Since $X_{m}$ has no affine points, $\mathcal{R}$ defines an isomorphic copy of $\mathcal{B}$ by Statement (i) of Theorem B.
\end{proof}

\section{Definability of $C^k$-smooth points}\label{section:ckdef}
\noindent Fix $k \in \N$ with $k \geq 2$. Throughout this section, let $X$ be a definable subset of $\R^n$. The main goal of this section is the following proposition.

\begin{Prop}
\label{prop:defineck}
The set of $C^k$-smooth points of $X$ is definable.
Equivalently: if $Y$ is a subset of $\R^n$, then the set of $C^k$-smooth points of $Y$ is $(\R,<,+,Y)$-definable.
\end{Prop}

\noindent It is worth re-iterating that we are not assuming $\mathcal{R}$ to be type A.
Let $x\in X$ be a $C^k$-smooth point of $X$. An application of the inverse function theorem shows that, after possibly permuting coordinates, there is $d\in \{0,\dots,n\}$, an open box $V \subseteq \R^d$ and a $C^k$-function $f : V \to \R^{n - d}$ such that $\Gamma(f) = U \cap X$ for some open box $U \subseteq \R^n$.
Note that such an $f$ must be definable.\newline

\noindent We will use the following consequence of the fact that the image of a $C^k$-submanifold under a $C^k$-diffeomorphism is again a $C^k$-submanifold.

\begin{Fact} Let $X,Y \subseteq \R^n$ be definable, let $x\in X$, and let $\tau: X \to Y$ be a definable $C^k$-diffeomorphism. Then $x$ is a $C^k$-smooth point of $X$ if and only if $\tau(x)$ is a $C^k$-smooth point of $Y$.
\end{Fact}

\subsection{When $\mathcal{R}$ is not field-type} Suppose that $\mathcal{R}$ is not field-type. We now give a proof of Proposition \ref{prop:defineck} in this case.
We need the following fact, which is a basic analysis exercise.


\begin{Fact}\label{fact:convex}
Let $f : W \to \R^m$ be a continuous function on a convex subset $W$ of $\R^d$.
Then $f$ is affine if and only if
$$ f \left( \frac{ x + y }{2} \right) = \frac{ f(x) + f(y) } {2} \quad \text{for all } x,y \in W. $$
\end{Fact}

\begin{proof}[Proof of Proposition \ref{prop:defineck} when $\mathcal{R}$ is not of field-type]
Let $X\subseteq \R^n$. We show that the set of $C^k$-smooth points of $X$ is equal to the set of affine points of $X$, and that the set of affine points of $X$ is definable. The latter statement follows immediately from Fact \ref{fact:convex}.
\newline

\noindent It is enough to show that every $C^k$-smooth point of $X$ is an affine point of $X$. Let $x\in X$ be a $C^k$-smooth point of $X$. Thus there is $d\in \{0,\dots,n\}$, open boxes $U\subseteq \R^n$ and $V\subseteq \R^d$ and a definable $C^k$-function $f : V \to \R^{n-d}$ such that after permuting coordinates if necessary, $x\in U$ and $\Gamma(f)=U\cap X$. Then $f$ is affine by Fact \ref{thm:old}(5), and hence $x$ is an affine point of $X$.\end{proof}

\subsection{When $\mathcal{R}$ is field-type} We now treat the case when $\mathcal{R}$ is field-type.
The following Fact~\ref{fact:defineCk2} follows using the usual definition of differentiability.

\begin{Fact}\label{fact:defineCk2}
Let $\mathcal{L}$ be the language of ordered rings, let $P$ be an $n$-ary predicate, and let $\mathcal{L}_P$ be the language $\mathcal{L}\cup \{P\}$. Then there is an $\mathcal{L}_P$-formula $\varphi(x)$ with $x = (x_1,\ldots,x_n)$ such that for every subset $X$ of $\R^n$
\[
\{ b \in \R^n : (\R,<,+,\cdot,X) \models \varphi(b) \}
\]
is the set of $C^k$-smooth points of $X$.
\end{Fact}

\begin{proof}[Proof of Proposition \ref{prop:defineck} when $\mathcal{R}$ is field-type]
We will first show that there are open intervals $I$ and $J$, definable functions $\oplus, \otimes : I^2 \to I$ and an isomorphism $\tau : I \to \R$ between $(I,<,\oplus,\otimes) \to (\R,+,\cdot)$ such that $J\subseteq I$ and the restriction of $\tau$ to $J$ is a $C^k$-diffeomorphism. By Theorem \ref{thm:old}.2, we only have to consider the case that $\mathcal{R}$ is type A or type C, because $\mathcal{R}$ is field-type.
The type A case is shown in \cite[Lemma 6.2]{HW-continuous}. \newline

\noindent Consider the case that $\mathcal{R}$ is type C. In this situation, $\mathcal{R}$ defines all bounded Borel sets. Set
\begin{align*}
\tau : (-2,2) &\to \R\\
x &\mapsto \left\{\begin{array}{lr}
        -\frac{1}{x+2}, & -2\leq x< -1\\
        x, & -1\leq n\leq 1\\
        \frac{1}{2-x}, & 1< n\leq 2
        \end{array}\right.
\end{align*}
Since $\mathcal{R}$ defines all bounded Borel sets, it is not hard see that $\mathcal{R}$ defines functions $\oplus$ and $\otimes$ such that $\tau$ is an isomorphism between $((-2,2),<,\oplus, \otimes)$ and \linebreak $(\R,+,\cdot)$. Now set $J:=(-1,1)$ and observe that the restriction of $\tau$ to $J$ is a $C^k$-diffeomorphism.\newline


\noindent  Set $L:=\tau(J)$. Let $\tau_m : J^m \to L^m$ be given by
$$ \tau_m (x_1,\ldots,x_m) = (\tau(x_1), \ldots, \tau(x_m)) \quad \text{for all} \quad x_1,\ldots,x_m \in J. $$
Note that $\tau_m$ is a $C^k$-diffeomorphism.\newline

\noindent Let $(Z_a)_{a \in S}$ be a definable family of subsets of $J^m$. We now show that the sets of $C^k$-smooth points of this family are uniformly definable. By Fact \ref{fact:defineCk2} there is a definable family $(Y_a)_{a\in S}$ such that
\[
Y_a := \{ z \in Z_a \ : \ \tau_m(z) \text{ is a $C^k$-smooth point of $\tau_m(Z_a)$}\}.
\]
Since $\tau_m$ is a $C^k$-diffeomorphism, we get that $Y_a$ is the set of $C^k$-smooth points of $Z_a$. \newline

%
%
%
%
%

\noindent Let $X\subseteq \R^n$ be definable. We are now ready to prove the definability of the set of $C^k$-smooth points of $X$. Fix $u \in J^m$ and $\varepsilon >0$ such that $B_{\varepsilon}(u)\subseteq J^m$. For $x\in X$ and $\delta >0$ with $\delta < \varepsilon$, we let  $g_{x} : B_{\delta}(x) \to B_{\delta}(u)$ be the $C^{\infty}$-diffeomorphism defined by
\[
y \mapsto u + (y-x).
\]
Now consider the set
\[
Y:= \{ x \in X \ : \ \exists \delta >0 \ \delta < \varepsilon \wedge u \text{ is a $C^k$-smooth point of } g_x(B_{\delta}(x)\cap X)\}.
\]
Observe that the $Y$ is definable by the argument in the previous paragraph. We finish the proof by establishing that $Y$ is the set of $C^k$-smooth points of $X$.\newline

\noindent Let $x\in X$ be a $C^k$-smooth point of $X$. Take $\delta \in(0,\varepsilon)$. Then $x$ is also a $C^k$-smooth point of $X\cap B_{\delta}(x)$. Since $g_x$ is a $C^{\infty}$-diffeomorphism, we have that $g_x(x)$ is a $C^k$-smooth point of $g_x(B_{\delta}(x)\cap X)$. Because $g_x(x)=u$, we have that $x\in Y$.\newline

\noindent Let $x \in Y$. 
Let $\delta \in (0,\varepsilon)$ be such that $u$ is a $C^k$-smooth point of $g_x(B_{\delta}(x)\cap X)$. Since $g_x(x)=u$ and $g_x$ is a $C^{\infty}$-diffeomorphism, we deduce that $x$ is a $C^k$-smooth point of $X\cap B_{\delta}(x)$ and hence a $C^k$-smooth point of $X$.
\end{proof}

\section{Type A expansions without dimension coincidence}
\label{subsection:dimension-failure}
\noindent
We now discuss two examples of type A expansions in which dimension coincidence fails.
First we present a natural family of type A expansions in which topological and Assouad dimensions do not agree, and then we describe a type A structure that defines a compact nowhere dense subset of $\R$ with positive Hausdorff dimension.

\subsection{$(\R,<,+,\Z)$ and related structures}
\label{section:top-assouad}
Perhaps the most example natural of a subset with topological dimension zero and positive Assouad dimension is $\Z$, which has Assouad dimension one.
More generally, $\Z^m$ has topological dimension zero and Assouad dimension $m$.
The structure $(\R,<,+,\Z)$ is well-known to be tame.
Miller~\cite{ivp} and Weispfenning~\cite{weis} independently showed that $(\R,<,+,\Z)$ has quantifier elimination in a natural expanded language.
It follows that $(\R,<,+,\Z)$ is locally o-minimal and NIP. Hence it is an example of type A expansions that does not define an isomorphic copy of $\mathcal{B}$. Similarly, Marker and Steinhorn showed in unpublished work that $(\R,<,+,\sin)$ is locally o-minimal, see Toffalori and Vozoris \cite{TV-local}.
Observe that $(\R,<,+,\sin)$ defines $\pi\Z$, a set of topological dimension zero and Assouad dimension one.
These two examples belong to a larger family of structures.

\medskip\noindent
Given $\alpha > 0$, let $+_\alpha :[0,\alpha)^2 \to [0,\alpha)$ be given by 
\[
t +_\alpha t^* := \begin{cases}
			 t + t^*, & \text{if $t + t^* < \alpha$}\\
            t + t^* - \alpha, & \text{otherwise.}
		 \end{cases}
\]
The following fact follows from Kawakami et al. \cite[Theorem 25]{MR2965427}.

\begin{fact}
\label{fact:az}
The following are equivalent:
\begin{enumerate}
\item $\Cal R$ is interdefinable with $(\R,<,+,\Cal C, \alpha\Z)$ for some $\alpha > 0$ and collection $\Cal C$ of bounded subsets of euclidean space such that $(\R,<,+,\Cal C)$ is o-minimal.
\item The reduct $\Cal R^*$ of $\Cal R$ generated by all bounded definable sets is o-minimal and $\Cal R$ is interdefinable with $(\Cal R^*,\alpha\Z)$ for some $\alpha > 0$.
\item $\Cal R$ is interdefinable with $(\Cal S,\alpha\Z)$ for some $\alpha > 0$ and o-minimal expansion $\Cal S$ of $(\R,<,+)$ such that $(\Cal S,\alpha\Z)$ is locally o-minimal.
\item There is $\alpha > 0$ and an o-minimal expansion $\Cal O$ of $([0,\alpha),<,+_\alpha)$ such that every definable subset of $\R^m$ is a finite union of sets of the form $X + Y$ for $\Cal O$-definable $X \subseteq [0,\alpha)^m$ and $(\alpha\Z,<,+)$-definable $Y \subseteq (\alpha\Z)^m$.
\end{enumerate}
\end{fact}

\noindent
For example, $(\R,<,+,\sin)$ is interdefinable with $(\R,<,+,\sin|_{[0,\pi]}, \pi\Z)$,  and \linebreak $(\R,<,+,\sin|_{[0,\pi]})$ is o-minimal.
More generally, if $f : \R \to \R$ is analytic and periodic with period $\alpha > 0$,  then $(\R,<,+,f)$ is interdefinable with $(\R,<,+,f|_{[0,\alpha]},\alpha\Z)$ and $(\R,<,+,f|_{[0,\alpha]})$ is o-minimal. Hence $(\R,<,+,f)$ satisfies (1)-(4) above.

\medskip\noindent
Suppose that $\Cal R$ satisfies one of the the equivalent conditions of Fact~\ref{fact:az}.
We assume that $\alpha = 1$ for the sake of simplicity.
Item (4) shows that $\Cal R$ is bi-interpretable with the disjoint union of an o-minimal expansion of $([0,1),<,+_1)$ and $(\Z,<,+)$.
In particular, every $\Cal R$-definable subset of $\Z^m$ is definable in $(\Z,<,+)$.
There is a canonical dimension on $(\Z,<,+)$-definable sets which we denote by $\dim_\Z$, see Cluckers~\cite{cluckers-presburger}.
Suppose that $X$ is a definable subset of $\R^m$.
We define the global dimension
\[
\dim_\mathrm{Global} X = \dim_\Z \{ b \in \Z^m : (b + [0,1)^m) \cap X \ne \emptyset \}
\]
and the local dimension
\[
\dim_\mathrm{Local} X = \max \{ \dim X \cap (b + [0,1)^m) : b \in \Z^m \}
\]
of $X$.
It is easy to see that the topological, packing, and Hausdorff dimensions of $X$ agree with $\dim_\mathrm{Local} X$.
One can also use the theory of weak tangents to show that the Assouad dimension of $X$ agrees with $\max \{ \dim_\mathrm{Local} X, \dim_\mathrm{Global} X \}$, see Fraser \cite[5.1]{Fraser2020} for a discussion of weak tangents.

\subsection{Topological and Hausdorff}
\label{section:top-hausdorff}
Whenever $\mathcal{R}$ is type B, we know that $\mathcal{R}$ defines an isomorphic copy of $\mathcal{B}$ by Fact \ref{thm:old}(1). In this section we show that in Statement (ii) of Theorem B the statement ``defines an isomorphic copy of $\mathcal{B}$'' cannot be replaced by the stronger statement ``is type B'', even under the stronger assumption that $\dim X < \dim_{\text{Hausdorff}} X$. We do so by giving an example of a type A expansion that defines a compact subset of $\mathbb{R}$ with topological dimension 0 and positive Hausdorff dimension, and hence positive packing dimension.
Our construction is based on the following reformulation of a result by Friedman and Miller~\cite[Theorem A]{FM-sparse}.

\begin{Fact}  
\label{fact:FM}
Let $\mathcal{M}$ be an o-minimal expansion of $(\R,<,+)$, and let $E \subseteq \mathbb{R}$ be closed and nowhere dense.
Then the following are equivalent:
\begin{enumerate}
\item $(\mathcal{M},E)$ is type A,
\item every $(\mathcal{M},E)$-definable subset of $\mathbb{R}$ has interior or is nowhere dense,
\item $f(E^n)$ is nowhere dense for every $\mathcal{M}$-definable $f : \mathbb{R}^n \to \mathbb{R}$.
\end{enumerate}
\end{Fact}

\begin{proof}
By \cite[Theorem A]{FM-sparse} we know that (3) implies (2).
A dense $\omega$-orderable set is dense and co-dense in some nonempty open interval. Therefore (2) implies (1).
It follows from Fact~\ref{prop:project1} that (1) implies (3).
\end{proof}

\noindent We also need the following immediate consequence of semi-linear cell decomposition (see van den Dries \cite[Corollary 1.7.8]{Lou}) for $(\R,<,+)$-definable sets.

\begin{Fact}\label{lem:semilinear}
Let $E \subseteq \mathbb{R}$.
The following are equivalent:
\begin{enumerate}
\item $f(E^n)$ is nowhere dense for every $(\R,<,+)$-definable $f : \mathbb{R}^n \to \mathbb{R}$,
\item $T(E^n)$ is nowhere dense for every $\mathbb{Q}$-linear $T : \mathbb{R}^n \to \mathbb{R}$.
\end{enumerate}
\end{Fact}

\noindent We now characterize compact nowhere dense subsets $E$ of $\mathbb{R}$ such that \linebreak $(\R,<,+,E)$ is type A.

\begin{Thm}\label{thm:Eclass}
Let $E \subseteq \mathbb{R}$ be compact and nowhere dense.
Then the following are equivalent:
\begin{enumerate}
\item $(\R,<,+,E)$ is type A,
\item every $(\R,<,+,E)$-definable subset of $\mathbb{R}$ has interior or is nowhere dense,
\item $T(E^n)$ is nowhere dense for every $\mathbb{Q}$-linear $T : \mathbb{R}^n \to \mathbb{R}$,
\item the subgroup of $(\mathbb{R},+)$ generated by $E$ is not equal to $\mathbb{R}$.
\end{enumerate}
\end{Thm}

\begin{proof}
The equivalence of (1), (2), and (3) follows immediately from Fact~\ref{fact:FM} and Fact~\ref{lem:semilinear}.
We show that (3) and (4) are equivalent.\newline

\noindent Suppose (3) holds. For $u = (u_1,\ldots,u_n) \in \mathbb{Z}^n$, set
\[
E_{u} := \{ u_1 e_1 + \ldots + u_n e_n \ : \ e_1,\ldots,e_n \in E \}.
\]
Then $\bigcup_{n > 0} \bigcup_{u \in \mathbb{Z}^n} E_{u}$ is the subgroup of $(\mathbb{R},+)$ generated by $E$.
Since $T(E^n)$ is nowhere dense for all $\mathbb{Q}$-linear $T : \mathbb{R}^n \to \mathbb{R}$, we have that $E_{u}$ is nowhere dense for every $u\in \Z^n$. Thus $\bigcup_{n > 0} \bigcup_{u \in \mathbb{Z}^n} E_{u}$ is a countable union of nowhere dense sets, and hence not equal to $\R$ by the Baire Category Theorem. Hence the subgroup of $(\mathbb{R},+)$ generated by $E$ is not equal to $\R$.\newline

\noindent Suppose (3) fails. Then there is $q = (q_1,\ldots,q_n) \in \mathbb{Q}^n$ and a $\mathbb{Q}$-linear $T: \R^n \to \R$ such that $T(E^n)$ is somewhere dense and for all $x = (x_1,\ldots,x_n) \in \mathbb{R}^n$
\[
 T(x) = q_1 x_1 + \ldots + q_n x_n.
\]
Then $T(E^n)$ is compact, because $E^n$ is compact. Thus $T(E^n)$ has interior.
Let $m \in \mathbb{Z}$ be such that $mq_i \in \mathbb{Z}$ for all $i \in \{1,\dots,n\}$, and set $mq := (mq_1,\ldots,mq_n)$.
Then
$$ E_{mq} = \{ mq_1 e_1 + \ldots + mq_n e_n : e_1,\ldots,e_n \in E\} = m T(E^n)$$
has interior.
So the subgroup of $(\mathbb{R},+)$ generated by $E$ has interior and therefore equals $\mathbb{R}$.
\end{proof}

\noindent When $E$ is closed, the subgroup of $(\mathbb{R},+)$ generated by $E$ is $F_\sigma$, and therefore Borel.
There are examples, first due to Erd\"os and Volkmann \cite{Erdos-Volkmann}, of proper Borel subgroups of $(\mathbb{R},+)$ with positive Hausdorff dimension.
If $G$ is a proper Borel subgroup of $(\mathbb{R},+)$ with positive Hausdorff dimension, then inner regularity of Hausdorff measure (see \cite[Corollary 4.5]{Mattila}) yields a compact subset of $G$ with positive Hausdorff dimension.
Such subsets necessarily have empty interior and have therefore topological dimension $0$. See Falconer \cite[Example 12.4]{Falconer-book} for specific examples of compact subsets of $\R$ that generate proper subgroups of ($\R,+)$ with positive Hausdorff dimension. By adding one of these compact set with positive Hausdorff dimension to $(\R,<,+)$, we obtain by Theorem \ref{thm:Eclass} a type A expansion that satisfies the assumptions of Theorem B (ii).

\bibliographystyle{amsplain}
\bibliography{Ref}

\end{document}